\documentclass[12pt,a4paper]{amsart}
\usepackage{amsmath,amssymb,amsthm}
\usepackage{verbatim}

\def\Cmath{\mathbb{C}}

\def\Kmath{\mathbb{K}}

\newtheorem{proposition}{Proposition}
\newtheorem{theorem}{Theorem}
\newtheorem{corollary}{Corollary}
\newtheorem{definition}{Definition}
\newtheorem{example}{Example}
\newtheorem{remark}{Remark}
\newcommand{\ilim}{\mathop{\varprojlim}\limits}

\title{About Leibniz cohomology and deformations of Lie algebras}

\author{A. Fialowski}
\address{Institute of Mathematics\\
E\"otv\"os Lor\'and University\\
 P\'azm\'any P\'eter s\'et\'any 1/C\\
 H-1117 Budapest, Hungary}
\email{fialowsk@cs.elte.hu}
\thanks{The research of the first author was partially supported by
OTKA grants K77757 and NK72523. The third author thanks the
Luxembourgian NRF for support via AFR grant PDR-09-062.}
\author{L. Magnin}
\address{Institut de Math\'{e}matiques de Bourgogne\\
UMR CNRS 5584, Universit\'{e} de Bourgogne, BP 47870,\\
21078 Dijon Cedex, France}
\email{magnin@u-bourgogne.fr}
\author{A. Mandal}
\address{University of Luxembourg, Campus Kirchberg\\
Math. Research Unit\\
6, rue Richard Coudenhove-Kelergi\\
L-1359 Luxembourg City}
\email{ashis.mandal@uni.lu}

\begin{document}
\maketitle
\begin{abstract}
We compare the second adjoint and trivial Leibniz
cohomology spaces  of a Lie algebra to the usual ones by a very elementary
 approach. The comparison
gives some conditions, which are easy to verify for a given Lie
algebra, for deciding whether it has more Leibniz deformations than just
the Lie ones. We also give the complete description of a Leibniz (and
Lie)
versal deformation of the 4-dimensional diamond Lie algebra, and study the
 case of its 5-dimensional analogue.
\end{abstract}
\section{Introduction}

Leibniz algebras, along with their Leibniz cohomologies,
were introduced in \cite{loday} as a non antisymmetric version of Lie algebras.
Lie algebras are special Leibniz algebras, and
Pirashvili  introduced \cite{pirashvili} a spectral sequence, that, when applied
to Lie algebras,
measures the  difference between the Lie  algebra cohomology
and the Leibniz cohomology.
Now, Lie algebras have deformations as Leibniz algebras
and those are piloted by the adjoint Leibniz 2-cocycles. In the present paper, we focus
on the second Leibniz cohomology groups
$HL^2(\mathfrak{g},\mathfrak{g}),$
$HL^2(\mathfrak{g},\Cmath)$
for adjoint and trivial
representations of a complex Lie algebra
$\mathfrak{g}$.
We  adopt a very elementary approach,
not resorting to  the Pirashvili sequence, to compare
$HL^2(\mathfrak{g},\mathfrak{g})$ and $HL^2(\mathfrak{g},\Cmath)$
to
$H^2(\mathfrak{g},\mathfrak{g})$ and $H^2(\mathfrak{g},\Cmath)$ respectively.
In both cases,
$HL^2$ appears to be the direct sum of 3 spaces:
$H^2 \oplus ZL^2_0 \oplus \mathcal{C}$ where
$H^2$ is the Lie algebra cohomology group,
$ZL^2_0$ is the space of symmetric
Leibniz-2-cocycles
and $\mathcal{C}$ is a space of \textit{coupled}
Leibniz-2-cocycles the nonzero elements of which have the property that
their symmetric and antisymmetric
parts are not Leibniz cocycles.
Our comparison gives some useful practical information about the
structure of Lie and Leibniz cocycles.
We analyse the case of Heisenberg algebras, the 4-dimensional
diamond algebra and its 5-dimensional analogue. We completely describe
a versal Leibniz and Lie deformation of the diamond algebra.

\section{Leibniz cohomology and deformations}
Recall that a (right) Leibniz algebra is an algebra
$\mathfrak{g}$
with a (non necessarily antisymmetric) bracket,
such that the right adjoint operations $[\cdot,Z]$ are required to be derivations
for any $Z \in \mathfrak{g}.$
In the presence of antisymmetry, that is equivalent to the Jacobi identity, hence any Lie
algebra is a Leibniz algebra.

The Leibniz cohomology
$HL^\bullet(\mathfrak{g},\mathfrak{g})$
of a Leibniz algebra is defined from the complex
$CL^\bullet(\mathfrak{g},\mathfrak{g}) =  \text{Hom }
\left(\mathfrak{g}^{\otimes \bullet}, \mathfrak{g} \right) =
\mathfrak{g} \otimes \left(\mathfrak{g}^*\right)^{\otimes \bullet}$
with the Leibniz-coboundary $\delta$ defined for $\psi \in
CL^n(\mathfrak{g},\mathfrak{g})$ by
\begin{multline*}
(\delta \psi)(X_1,X_2, \cdots , X_{n+1}) =
\\
 [X_1,\psi(X_2,\cdots, X_{n+1})]
 +
 \sum_{i=2}^{n+1} \,(-1)^i [\psi(X_1, \cdots, \hat{X_i}, \cdots, X_{n+1}), X_i]
 \\
 +
 \sum_{1\leqslant i < j \leqslant n+1} (-1)^{j+1}  \,
 \psi(X_1, \cdots, X_{i-1},[X_i,X_j],X_{i+1},\cdots, \hat{X_j}, \cdots, X_{n+1}).
\end{multline*}
(If $\mathfrak{g}$ is a Lie algebra, $\delta$ coincides with the usual coboundary $d$
on $C^\bullet(\mathfrak{g},\mathfrak{g}) =
\mathfrak{g} \otimes \bigwedge^{\bullet} \, \mathfrak{g}^*.$
)

For $\psi \in
CL^1(\mathfrak{g},\mathfrak{g}) =
C^1(\mathfrak{g},\mathfrak{g}) =
\mathfrak{g} \otimes {\mathfrak{g}}^*$
$$(\delta \psi)(X,Y) = [X,\psi(Y)] +[\psi(X),Y] -\psi([X,Y]).$$

For $\psi \in
CL^2(\mathfrak{g},\mathfrak{g}) =
\mathfrak{g} \otimes \left(\mathfrak{g}^*\right)^{\otimes 2},$
\begin{multline*}
(\delta \psi)(X,Y,Z) = [X,\psi(Y,Z)] +[\psi(X,Z),Y] - [\psi(X,Y),Z]
\\
-\psi([X,Y],Z)
+\psi(X,[Y,Z])
+\psi([X,Z],Y).
\end{multline*}

In the same way,
the trivial Leibniz cohomology
$HL^\bullet(\mathfrak{g},\Cmath)$
is defined from the complex
$CL^\bullet(\mathfrak{g},\Cmath) =
\left(\mathfrak{g}^*\right)^{\otimes \bullet}$
with the trivial-Leibniz-coboundary $\delta_\Cmath$ defined for $\psi \in
CL^n(\mathfrak{g},\Cmath)$ by
\begin{multline*}
(\delta_\Cmath \psi)(X_1,X_2, \cdots , X_{n+1}) =
\\
 \sum_{1\leqslant i < j \leqslant n+1} (-1)^{j+1}  \,
 \psi(X_1, \cdots, X_{i-1},[X_i,X_j],X_{i+1},\cdots, \hat{X_j}, \cdots, X_{n+1}).
\end{multline*}
If $\mathfrak{g}$ is a Lie algebra, $\delta_\Cmath$ is
the usual coboundary $d_\Cmath$
on $C^\bullet(\mathfrak{g},\Cmath) =
\bigwedge^{\bullet} \, \mathfrak{g}^*.$

For $\psi \in
CL^1(\mathfrak{g},\Cmath) =
\mathfrak{g}^* ,$
$$(\delta_{\Cmath} \psi)(X,Y) = -\psi([X,Y]).$$

For $\psi \in
CL^2(\mathfrak{g},\Cmath) =
\left(\mathfrak{g}^*\right)^{\otimes 2},$
\begin{equation*}
(\delta_\Cmath \psi)(X,Y,Z) =
-\psi([X,Y],Z)
+\psi(X,[Y,Z])
+\psi([X,Z],Y).
\end{equation*}

For computing Leibniz deformations, we need to consider the 2- and
3-dimensional cohomology cocycles.

Let $\mathbb{K}$ be a field of zero characteristic. We recall the notion of deformation
 of a Lie (Leibniz)
 algebra $\mathfrak{g}$ ($L$) over a commutative algebra base $A$ with
 identity, with a fixed augmentation $\varepsilon:{A}\rightarrow
\mathbb{K}$ and maximal ideal $\mathfrak{M}$.
 Assume $\dim(\mathfrak{M}^k/\mathfrak{M}^{k+1})<\infty$
for every $k$ (see \cite{f2, fmm}).
\begin{definition}
A deformation $\lambda$ of a Lie algebra $\mathfrak{g}$ (or a Leibniz algebra ${L}$) with base
$({A},\mathfrak{M})$, or simply with base ${A}$ is an $A$-Lie algebra (or an ${A}$-Leibniz
algebra) structure on the tensor product
${A}\otimes {\mathfrak {g}}$ (or $A\otimes L$) with the bracket $[,]_\lambda$ such that
 \[
 \varepsilon\otimes id:{A}\otimes {\mathfrak{g}}\rightarrow \mathbb{K}\otimes {\mathfrak{g}}~~ (\mbox{or}~\varepsilon\otimes id:{A}\otimes {L}\rightarrow \mathbb{K}\otimes {L})
 \]
 is an $A$-Lie algebra (${A}$-Leibniz algebra)  homomorphism.
\end{definition}
A deformation of the Lie (Leibniz) algebra $\mathfrak{g}$ ($L$) with base $A$ is called {\it infinitesimal}, {or \it first order}, if in addition to this $\mathfrak{M}^2=0$. We call a deformation of {\it order k}, if $\mathfrak{M}^{k+1}=0$. A deformation with base is called local if $A$ is a local algebra over $\mathbb{K}$, which means $A$ has a unique maximal ideal.

Suppose $A$ is a complete local algebra ( $A=\ilim_{n\rightarrow\infty}
({A}/{\mathfrak{M}^n})$), where $\mathfrak{M}$ is the maximal
ideal in $A$. Then a deformation of $\mathfrak{g}$ ($L$) with base $A$ which is obtained as the projective limit of deformations of $\mathfrak{g}$ $(L)$ with base $A/\mathfrak{M}^{n}$ is called a {\it formal deformation} of $\mathfrak{g}$ $(L)$.

\begin{definition}(see \cite{f2})
Let $C$ be a complete local algebra. A formal deformation $\eta$ of a Lie algebra $\mathfrak{g}$ (Leibniz algebra $L$) with base $C$ is called versal, if\\
(i)~for any formal deformation $\lambda$ of $\mathfrak{g}$ ($L$) with  base $A$ there exists a homomorphism $f:C \rightarrow A$ such that the deformation $\lambda$ is equivalent to $f_{*}\eta$; \\
(ii)~if $A$ satisfies the condition ${\mathfrak{M}}^2=0$, then $f$ is unique.
\end{definition}

\begin{theorem}(\cite{f2, fmm})
If $H^2(\mathfrak{g};\mathfrak{g})$ is finite dimensional, then there
exists a versal deformation of $\mathfrak{g}$ (similarly for $L$).
\end{theorem}

In \cite{f1} a construction for a versal deformation of a Lie algebra
was given and it was generalized to Leibniz algebras in \cite{fmm}. The
computation for a specific Leibniz algebra example was given in
\cite{fm}.

\section{Comparison of the cohomology spaces $HL^2$ and $H^2$ for a Lie algebra}

In \cite{pirashvili} the relation between Chevalley-Eilenberg
and Leibniz homology with coefficients in a right module is considered via
spectral sequence. The statements are valid in cohomological version as
well. As a corollary, one deduces
\begin{proposition}\cite{pirashvili}
Let $\mathfrak{g}$ be a Lie algebra over a field $\Kmath$ and $M$ be a right
$\mathfrak{g}$-module. If\\
 $$H_*(\mathfrak{g},M) = 0, \text{ \ then \ } HL_*(\mathfrak{g},M) = 0.$$
\end{proposition}

As the similar statement is true for cohomologies, it implies that
rigid Lie algebras are Leibniz rigid as well.

\medskip
Now we describe the Leibniz 2-cohomology spaces with the help of Lie
2-cohomology space of a Lie algebra $\mathfrak{g}$.\\


Recall that a  symmetric bilinear form
$B \in  S^2 \mathfrak{g}^*$
is invariant, i.e.
$B \in \left(S^2 \mathfrak{g}^*\right) ^{\mathfrak{g}}$
if and only if
$B([Z,X],Y) =-B(X,[Z,Y])  \; \forall X,Y,Z \in
\mathfrak{g}.$
The Koszul map \cite{koszul}
$\mathcal{I} \, : \,
\left(S^2 \mathfrak{g}^*\right) ^{\mathfrak{g}}
\rightarrow
\left(\bigwedge^3 \mathfrak{g}^*\right)^{\mathfrak{g}}
\subset Z^3(\mathfrak{g},\Cmath)
$
is defined by $\mathcal{I}(B)=  I_B,$ with
$I_B(X,Y,Z)=  B([X,Y],Z) \; \forall X,Y,Z \in  \mathfrak{g}.$
Since the projection
$\pi \; : \; \mathfrak{g} \rightarrow \mathfrak{g}/\mathcal{C}^2 \mathfrak{g}$
induces an isomorphism
$$\varpi \; : \; \ker{\mathcal{I}} \rightarrow
S^2  \left( \mathfrak{g}/\mathcal{C}^2 \mathfrak{g} \right)^*
,$$
(where $\mathcal{C}^2 \mathfrak{g}=[\mathfrak{g},\mathfrak{g}])$,
$\dim
\left(S^2 \mathfrak{g}^*\right) ^{\mathfrak{g}}
= \frac{p(p+1)}{2}
+\dim \text{Im\,}{\mathcal{I}},$
with $p = \dim
 H^1(\mathfrak{g},\Cmath).$
For reductive
$\mathfrak{g},$
$\dim
\left(S^2 \mathfrak{g}^*\right) ^{\mathfrak{g}}
=
\dim H^3(\mathfrak{g},\Cmath)$
.
Note also that the restriction of
$\delta_\Cmath$
to $\left(S^2 \mathfrak{g}^*\right) ^{\mathfrak{g}}$
is $-\mathcal{I}.$

\begin{definition}
$\mathfrak{g}$ is said to be
 $\mathcal{I}$-null
 (resp.
 $\mathcal{I}$-exact)
if
$\mathcal{I}=0$
 (resp.
$\text{Im\,}{\mathcal{I}} \subset B^3(\mathfrak{g},\Cmath))$.
\end{definition}
For more details on $\mathcal{I}$-null Lie algebras, see \cite{notesmagnin}.
\begin{example}
\label{exampleheisenberg}
\rm
The $(2N+1)$-dimensional complex Heisenberg Lie algebra
 $\mathcal{H}_N$
($N\geqslant 1$)
with basis $(x_i)_{1\leqslant i \leqslant 2N+1}$ and
nonzero commutation relations (with anticommutativity)
$[x_i,x_{N+i}] =x_{2N+1}$
$(1 \leqslant i \leqslant N)$ is
$\mathcal{I}$-null since, for any
$B \in \left(S^2 {\mathcal{H}_N}^*\right) ^{\mathcal{H}_N},$
$B(x_i,x_{2N+1})=B(x_i,[x_i,x_{N+i}]) = -B([x_i,x_i], x_{N+i})=0$
(similarly with $x_{N+i}$ instead of $x_i$)
$(1 \leqslant i \leqslant N),$
and $B(x_{2N+1},x_{2N+1})=
B(x_{2N+1},[x_1,x_{N+1}]) =
-B([x_1,x_{2N+1}], x_{N+1})=0.$
\end{example}

If $\mathfrak{c}$ denotes the center of
$\mathfrak{g}$,
$\mathfrak{c} \,  \otimes \left(S^2 \mathfrak{g}^*\right) ^{\mathfrak{g}}$
is the space of
invariant $\mathfrak{c}$-valued symmetric bilinear map and
we denote
$F= Id\, \otimes \mathcal{I}: \,
\mathfrak{c} \,  \otimes \left( S^2 \mathfrak{g}^*\right) ^{\mathfrak{g}}
\rightarrow C^3(\mathfrak{g}, \mathfrak{g})=
\mathfrak{g} \,  \otimes \bigwedge^3 \mathfrak{g}^*.$
Then
$\text{Im\,}F= \mathfrak{c}\, \otimes \text{Im\,}{\mathcal{I}}.$

\begin{theorem}
\label{theo2}
Let $\mathfrak{g}$ be any finite dimensional complex Lie algebra and
$ZL^2_0(\mathfrak{g},\mathfrak{g})$
(resp. $ZL^2_0(\mathfrak{g},\Cmath))$
the space of
symmetric adjoint (resp. trivial) Leibniz 2-cocycles.
\\(i)
$ZL^2(\mathfrak{g},\mathfrak{g}) \left/
\left(Z^2(\mathfrak{g},\mathfrak{g}) \oplus
ZL_0^2(\mathfrak{g},\mathfrak{g})\right) \right. \cong
\left(\mathfrak{c}\, \otimes \text{Im\,}{\mathcal{I}}\right) \cap
B^3(\mathfrak{g},\mathfrak{g}) .$
\\(ii)
$ZL^2_0(\mathfrak{g},\mathfrak{g})= \mathfrak{c}\, \otimes \ker{\mathcal{I}}.$
In particular,
$\dim ZL^2_0(\mathfrak{g},\mathfrak{g}) = c \, \frac{p(p+1)}{2} $
where
$c=\dim  \mathfrak{c} $ and
$p=
\dim  \mathfrak{g}/\mathcal{C}^2 \mathfrak{g}
=\dim H^1( \mathfrak{g},\Cmath).$
 \smallskip
\\(iii)
$HL^2(\mathfrak{g},\mathfrak{g}) \cong H^2(\mathfrak{g},\mathfrak{g}) \oplus
\left(\mathfrak{c} \otimes \ker{\mathcal{I}}\right)
\oplus \left(
\left(\mathfrak{c} \otimes \text{Im\,}{\mathcal{I}}\right) \cap B^3(\mathfrak{g},\mathfrak{g}) \right).$
\\(iv)
$ZL^2(\mathfrak{g},\Cmath) \left/
\left(Z^2(\mathfrak{g},\Cmath) \oplus ZL_0^2(\mathfrak{g},\Cmath)\right)
\right. \cong
\text{Im\,}{\mathcal{I}} \cap B^3(\mathfrak{g},\Cmath) .$
\\(v)
$ZL^2_0(\mathfrak{g},\Cmath)= \ker{\mathcal{I}}.$
\\(vi)
$HL^2(\mathfrak{g},\Cmath) \cong H^2(\mathfrak{g},\Cmath) \oplus
\ker{\mathcal{I}}
\oplus
\left(\text{Im\,}{\mathcal{I}} \cap B^3(\mathfrak{g},\Cmath) \right).$
 \smallskip
\end{theorem}
\begin{proof}
(i) The Leibniz 2-cochain space
$CL^2(\mathfrak{g},\mathfrak{g}) =
\mathfrak{g}
\otimes
\left(\mathfrak{g}^*\right)^{\otimes 2}
$
decomposes as
$\left(
\mathfrak{g}
\otimes
\bigwedge^2 \mathfrak{g}^*  \right) \;
 \oplus \; \left(
\mathfrak{g}
\otimes
 S^2 \,\mathfrak{g}^*  \right)$
 with
 $
\mathfrak{g}
\otimes
 S^2 \,\mathfrak{g}^*  $
 the space of symmetric elements in
$CL^2(\mathfrak{g},\mathfrak{g}).$
By definition of the Leibniz coboundary $\delta$, one has for
$\psi \in CL^2(\mathfrak{g},\mathfrak{g})$ and $X,Y,Z \in \mathfrak{g}$
\begin{equation}
\label{leibnizcocycle}
(\delta \psi)(X,Y,Z) = u+v+w+r+s+t
\end{equation}
with $u=[X,\psi(Y,Z)],\;
v=[\psi(X,Z),Y], \;
w=-[\psi(X,Y),Z],\;
r=-\psi([X,Y],Z), \;
$
$s=\psi(X,[Y,Z]),  \;
t=\psi([X,Z],Y).$
$\delta$ coincides with the usual coboundary operator on
$
\mathfrak{g}
\otimes
\bigwedge^2 \mathfrak{g}^*   .$
Now, let $\psi = \psi_1 + \psi_0 \in
CL^2(\mathfrak{g},\mathfrak{g})$
, $\psi_1 \in
\mathfrak{g}
\otimes
\bigwedge^2 \mathfrak{g}^*   , \;
 \psi_0 \in
\mathfrak{g}
\otimes
 S^2 \,\mathfrak{g}^*  .$
\\
\par
 Suppose
$\psi \in ZL^2(\mathfrak{g},\mathfrak{g}):$
$\delta \psi = 0 = \delta\psi_1 + \delta \psi_0 =  d \psi_1 + \delta \psi_0.$
Then $\delta \psi_0 = -d \psi_1
\in
\mathfrak{g}
\otimes
\bigwedge^3 \mathfrak{g}^*  $ is antisymmetric.
Then permuting $X$ and $Y$ in formula
(\ref{leibnizcocycle}) for $\psi_0$ yields
$(\delta \psi_0)(Y,X,Z) = -v - u + w - r + t + s.$
As
$\delta \psi_0$ is antisymmetric, we get
\begin{equation}
\label{first}
w + s + t = 0.
\end{equation}
Now, the circular permutation $(X,Y,Z)$ in
(\ref{leibnizcocycle}) for $\psi_0$ yields
$(\delta \psi_0)(Y,Z,X) = -v - w + u - s - t + r.$
Again, by antisymmetry,
\begin{equation}
\label{second}
v +w + s + t = 0,
\end{equation}
i.e.
$(\delta \psi_0)(X,Y,Z) = u + r.$
From (\ref{first}) and (\ref{second}), $v = 0.$
Applying twice
the circular permutation $(X,Y,Z)$ to $v$, we get
first $w=0$ and then $u=0.$
Hence $(\delta \psi_0)(X,Y,Z) = r
=-\psi_0([X,Y],Z).$
Note first that $u=0$ reads
$[X,\psi_0(Y,Z)] = 0.$
As $X,Y,Z$ are arbitrary, $\psi_0$ is
$\mathfrak{c}$-valued.
Now the permutation  of $Y$ and $Z$ changes
$r$ to $- t=s$ (from (\ref{second})).
Again, by antisymmetry
of $\delta \psi_0,$
$r=t=-s.$
As $X,Y,Z$ are arbitrary, one gets
$\psi_0 \in \mathfrak{c} \,  \otimes \left(S^2 \mathfrak{g}^*\right) ^{\mathfrak{g}}.$
Now $F(\psi_0)= - r = -\delta\psi_0 = d \psi_1 \in
B^3(\mathfrak{g},\mathfrak{g} ).$
Hence
$$\psi_0 \in ZL^2_0(\mathfrak{g},\mathfrak{g})
\Leftrightarrow
F(\psi_0)=0
\Leftrightarrow
\psi_1 \in Z^2(\mathfrak{g},\mathfrak{g})
\Leftrightarrow
\psi_0 \in \mathfrak{c} \otimes \ker{\mathcal{I}}.$$
Consider now the linear map
$\Phi \; : \,
ZL^2(\mathfrak{g},\mathfrak{g}) \rightarrow
F^{-1}(B^3(\mathfrak{g},\mathfrak{g})) \left/\; \ker{F} \right.$
defined by
$\psi \mapsto [\psi_0] \,(\text{mod} \ker{F}).$
$\Phi$ is onto: for any $[\varphi_0] \in
F^{-1}(B^3(\mathfrak{g},\mathfrak{g})) \left/\; \ker{F} \right.$
,
$\varphi_0 \in
\mathfrak{c}\, \otimes  \left(S^2\mathfrak{g^*}\right)^{\mathfrak{g}},$
one has $F(\varphi_0) \in B^3(\mathfrak{g},\mathfrak{g}),$
hence $F(\varphi_0)= d\varphi_1,$ $\varphi_1 \in C^2(\mathfrak{g},\mathfrak{g}),$
and then $\varphi=\varphi_0 +\varphi_1$ is a Leibniz cocycle such that
$\Phi( \varphi) =[\varphi_0].$
Now $\ker{\Phi}=
Z^2(\mathfrak{g},\mathfrak{g}) \oplus ZL_0^2(\mathfrak{g},\mathfrak{g}),$
since condition $[\psi_0] =[0]$ reads $\psi_0 \in \ker{F}$ which is equivalent
to $\psi \in
Z^2(\mathfrak{g},\mathfrak{g}) \oplus ZL_0^2(\mathfrak{g},\mathfrak{g}).$
Hence $\Phi$ yields an isomorphism
$ZL^2(\mathfrak{g},\mathfrak{g}) \left/
\left(Z^2(\mathfrak{g},\mathfrak{g}) \oplus ZL_0^2(\mathfrak{g},\mathfrak{g})\right)
\right. \cong
F^{-1}(B^3(\mathfrak{g},\mathfrak{g})) \left/ \;\ker{F} \right. .$
The latter is isomorphic to
$ \text{Im\,}F  \cap B^3(\mathfrak{g},\mathfrak{g})
\cong
\left(\mathfrak{c} \otimes \text{Im\,}{\mathcal{I}}\right) \cap B^3(\mathfrak{g},\mathfrak{g}).$
\\(ii)  Results from the invariance of
 $\psi_0 \in ZL^2_0(\mathfrak{g},\mathfrak{g}).$
\\(iii)
Results immediately from (i), (ii) since
 $BL^2(\mathfrak{g},\mathfrak{g})=B^2(\mathfrak{g},\mathfrak{g})$
 as the Leibniz differential on
$CL^1(\mathfrak{g},\mathfrak{g}) =
   \mathfrak{g}^* \otimes \mathfrak{g} = C^1(\mathfrak{g},\mathfrak{g})$
coincides with the usual one.
\\
(iv)-(vi) are similar.
\end{proof}

\begin{remark}
\rm
Since
$ \ker{\mathcal{I}}\oplus \left(\text{Im\,}{\mathcal{I}} \cap B^3(\mathfrak{g},\Cmath) \right) \cong \ker{h}$ where $h$ denotes $\mathcal{I}$ composed with the projection of
$Z^3(\mathfrak{g},\Cmath)$
onto $H^3(\mathfrak{g},\Cmath),$ the result (vi) is the same as
in \cite{hpl}.
\end{remark}

\begin{remark}
\rm
Any supplementary subspace to
$Z^2(\mathfrak{g},\Cmath) \oplus ZL_0^2(\mathfrak{g},\Cmath)$ in
$ZL^2(\mathfrak{g},\Cmath)$
consists of \textit{coupled} Leibniz 2-cocycles, i.e. the nonzero elements
have the property that their
symmetric and antisymmetric
parts are not cocycles.
To get such a supplementary subspace, pick any supplementary subspace $W$ to
$\ker{\mathcal{I}}$ in
$\left(S^2 \mathfrak{g}^*\right) ^{\mathfrak{g}}$
and take $\mathcal{C} =\left\{ B +\omega\, ; B \in W \cap {\mathcal{I}}^{-1}
(B^3(\mathfrak{g},\Cmath)), I_B=d\omega \right\}.$
\end{remark}
\begin{definition}
$\mathfrak{g}$ is said to be adjoint (resp. trivial) $ZL^2$-uncoupling
if
$$
\left(\mathfrak{c}\, \otimes \text{Im\,}{\mathcal{I}}\right) \cap
B^3(\mathfrak{g},\mathfrak{g}) =\{0\} \left( \text{resp. Im\,}{\mathcal{I}}
\cap B^3(\mathfrak{g},\Cmath) =\{0\}\right).
$$
\end{definition}

The class of adjoint
$ZL^2$-uncoupling  Lie algebras is rather extensive since it contains
all zero-center Lie algebras
and all $\mathcal{I}$-null Lie algebras.
For non zero-center, adjoint $ZL^2$-uncoupling implies trivial $ZL^2$-uncoupling,Adjoint $ZL^2$-uncoupling implies trivial $ZL^2$-uncoupling,
since
$\mathfrak{c}\, \otimes \left( \text{Im\,}{\mathcal{I}} \cap B^3(\mathfrak{g},\Cmath)\right) \subset
\left(\mathfrak{c}\, \otimes \text{Im\,}{\mathcal{I}}\right) \cap B^3(\mathfrak{g},\mathfrak{g}).$
The reciprocal holds obviously true
for $\mathcal{I}$-exact Lie algebras.
However we do not know
if it holds true in general
(e.g. we do not know of a nilpotent Lie algebra which is not
 $\mathcal{I}$-exact).

\begin{corollary}
(i)
$HL^2(\mathfrak{g},\mathfrak{g}) \cong H^2(\mathfrak{g},\mathfrak{g}) \oplus
\left(\mathfrak{c} \otimes \ker{\mathcal{I}}\right)$
if and only if
$\mathfrak{g}$ is adjoint $ZL^2$-uncoupling.
\\(ii)
$HL^2(\mathfrak{g},\Cmath) \cong H^2(\mathfrak{g},\Cmath) \oplus
\ker{\mathcal{I}}$
if and only if
$\mathfrak{g}$ is trivial $ZL^2$-uncoupling.
\end{corollary}
\par
\begin{corollary}
For any Lie algebra
$\mathfrak{g}$ with trivial center
$\mathfrak{c}= \{0\},$
 $HL^2(\mathfrak{g},\mathfrak{g})=H^2(\mathfrak{g},\mathfrak{g}).$
 \end{corollary}
\begin{remark}
This fact also follows from the cohomological version of Theorem A in
\cite{pirashvili}.
\end{remark}

\begin{proof}
Let $\mathfrak{g}$ be a Lie algebra and $M$ be a right
$\mathfrak{g}$-module. Consider the product map $m :
\mathfrak{g}\otimes\Lambda^n{\mathfrak{g}} \longrightarrow
\Lambda^{n+1}$ in the exterior algebra. This map yields an epimorphism
of chain complexes

$C_*(\mathfrak{g},\mathfrak{g}) \longrightarrow
C_i(\mathfrak{g},\Kmath)[-1]$,

where $C_*(\mathfrak{g},\Kmath)$ is the reduced chain complex:

$C_0(\mathfrak{g},\Kmath) = 0$,

$C_i(\mathfrak{g},\Kmath) = C_i(\mathfrak{g},\Kmath) \text{ \ for \ } i > 0$.
Define the chain complex $CR_*(\mathfrak{g})$ such that
$CR_*(\mathfrak{g}[1]$ is the kernel of the epimorphism
$C_*(\mathfrak{g}, \mathfrak{g}) \longrightarrow
C_*(\mathfrak{g},\Kmath)[-1]$. Denote the cohomology of $CR_*(\mathfrak{g})$
by $HR_*(\mathfrak{g})$.

Let us recall Theorem A in \cite{pirashvili}.\\
There exists a spectral sequence
$$
E^2_{pq} = HR_p(\mathfrak{g} \otimes HL_q(\mathfrak{g}, M)
\Longrightarrow  H^{rel}_{p+q}(\mathfrak{g}, M).
$$

\medskip
As the center of our Lie algebra is $0$, it follows that $E^2_{00}=0$,
and so we get $H^{rel}_0(\mathfrak{g},\mathfrak{g})=0$.

But then from the exact sequence in \cite{pirashvili}
$$
0 \leftarrow H_2(\mathfrak{g},M) \leftarrow HL_2(\mathfrak{g},M)
\leftarrow H^{rel}_0(\mathfrak{g},M) \leftarrow H_3(\mathfrak{g},M)
\leftarrow ...
$$

we get
$$
HL_2(\mathfrak{g},M) = H_2(\mathfrak{g},M).
$$
\end{proof}

\begin{corollary}
For any
reductive algebra
Lie
$\mathfrak{g}$ with  center
$\mathfrak{c},$
$HL^2(\mathfrak{g},\mathfrak{g}) =
H^2(\mathfrak{g},\mathfrak{g}) \oplus \,
 \left(\mathfrak{c} \, \otimes S^2 \mathfrak{c}^* \right),$
and
$\dim H^2(\mathfrak{g},\mathfrak{g}) =  \frac{c^2(c-1)}{2}$
 with $c=\dim \mathfrak{c}.$
\end{corollary}
\begin{proof}
$\mathfrak{g} = \mathfrak{s} \oplus \mathfrak{c}$
with $\mathfrak{s} = \mathcal{C}^2  \mathfrak{g}$ semisimple.
We first prove that
$\mathfrak{g}$ is adjoint $ZL^2$-uncoupling.
$ \mathfrak{c} \,  \otimes \left(S^2 \mathfrak{g}^*\right)^{\mathfrak{g}}
= \left( \mathfrak{c} \,  \otimes \left(S^2 \mathfrak{s}^*\right)
^{\mathfrak{s}} \right)
\oplus \left(\mathfrak{c} \,  \otimes S^2 \mathfrak{c}^*\right)
= c \, \left( S^2 \mathfrak{s}^*\right) ^{\mathfrak{s}}
\oplus c \, \left( S^2 \mathfrak{c}^*\right).$
Suppose first
 $\mathfrak{s} $
 simple. Then any bilinear symmetric invariant form on
 $\mathfrak{s} $   is some multiple of the Killing form $K.$
 Hence
$ \mathfrak{c} \,  \otimes \left(S^2 \mathfrak{g}^*\right)^{\mathfrak{g}}
= c \, (\Cmath K)
\oplus c \, \left( S^2 \mathfrak{c}^*\right).$
For any $\psi_0 \in  \mathfrak{c} \,  \otimes\left( S^2 \mathfrak{g}^*\right)^{\mathfrak{g}},$
$F({\psi_0})$ is then some linear combination of copies of $I_K.$
As is well-known, $I_K$  is no coboundary. Hence  if we suppose that
$F({\psi_0})$ is a coboundary, necessarily
$F({\psi_0})=0.$
$\mathfrak{g}$ is adjoint $ZL^2$-uncoupling
when $\mathfrak{s} $  is  simple.
Now, if
 $\mathfrak{s} $  is not simple,
 $\mathfrak{s} $ can be decomposed as a direct sum
 $ \mathfrak{s}_1 \oplus \cdots \oplus   \mathfrak{s}_m$ of simple ideals of
 $\mathfrak{s}. $
 Then
$  \left( S^2 \mathfrak{s}^*\right) ^{\mathfrak{s}} =
\bigoplus_{i=1}^{m} \,    \left( S^2 \mathfrak{s_i}^*\right) ^{\mathfrak{s_i}}
=
\bigoplus_{i=1}^{m} \,    \Cmath \, K_i$
($K_i $ Killing form of
 $\mathfrak{s}_i. $)
The same reasoning then applies and shows that
$\mathfrak{g}$ is adjoint $ZL^2$-uncoupling.
From (ii) in theorem \ref{theo2},
$ZL^2_0(\mathfrak{g},\mathfrak{g}) =  \mathfrak{c} \, \otimes S^2 \mathfrak{c}^* .$
Now,
$\mathfrak{g} = \mathfrak{s} \oplus \mathfrak{c}$
with $\mathfrak{s} = \mathcal{C}^2  \mathfrak{g}$ semisimple.
 $\mathfrak{s} $ can be decomposed as a direct sum
 $ \mathfrak{s}_1 \oplus \cdots \oplus   \mathfrak{s}_m$ of ideals of
 $\mathfrak{s} $
 hence of $\mathfrak{g}. $  Then
$H^2(\mathfrak{g},\mathfrak{g}) =
\bigoplus_{i=1}^{m} H^2(\mathfrak{g},\mathfrak{s}_i)\,  \oplus \,  H^2(\mathfrak{g},\mathfrak{c}).$
As
 $\mathfrak{s}_i $ is a nontrivial
$\mathfrak{g}$-module,
$H^2(\mathfrak{g},\mathfrak{s}_i) = \{0\}$ (\cite{guichardet}, Prop. 11.4, page 154).
Hence
$H^2(\mathfrak{g},\mathfrak{g}) =
H^2(\mathfrak{g},\mathfrak{c}) = c\, H^2(\mathfrak{g},\Cmath).$
By the K\"{u}nneth formula and
Whitehead's lemmas,
$H^2(\mathfrak{g},\Cmath) =
\left(H^2(\mathfrak{s},\Cmath) \otimes H^0(\mathfrak{c},\Cmath) \right) \, \oplus
\left(H^1(\mathfrak{s},\Cmath) \otimes H^1(\mathfrak{c},\Cmath) \right) \, \oplus
\left(H^0(\mathfrak{s},\Cmath) \otimes H^2(\mathfrak{c},\Cmath) \right)
= $
$H^0(\mathfrak{s},\Cmath) \otimes H^2(\mathfrak{c},\Cmath)
= \Cmath \otimes H^2(\mathfrak{c},\Cmath). $
Hence
\linebreak[4]
$\dim H^2(\mathfrak{g},\mathfrak{g}) =  \frac{c^2(c-1)}{2}.$
\end{proof}
\par

\section{Examples}

For $\omega, \pi \in  \mathfrak{g}^*,$ $\odot $
stands for  the symmetric product
 $\omega \odot \pi =  \omega \otimes \pi +  \pi \otimes  \omega.$
\begin{example}
\rm
For
$\mathfrak{g} = \mathfrak{gl}(n),$
$$HL^2(\mathfrak{g},\mathfrak{g}) =
ZL^2_0(\mathfrak{g},\mathfrak{g}) = \Cmath \,\left( x_{n^2} \oplus (
\omega^{n^2} \odot \omega^{n^2} )\right),$$  where
$(x_i)_{1\leqslant i \leqslant n^2}$ is a basis of
$\mathfrak{g}$
such that
$(x_i)_{1\leqslant i \leqslant n^2-1}$ is a basis of
$ \mathfrak{sl}(n)$ and $x_{n^2}$ is the identity matrix, and
$(\omega^i)_{1\leqslant i \leqslant n^2}$ the dual basis
to $(x_i)_{1\leqslant i \leqslant n^2}.$
Hence there is a
unique Leibniz deformation of
$\mathfrak{gl}(n).$
\end{example}

\begin{corollary}
Let $\mathfrak{g}=\mathcal{H}_N$ be the $(2N+1)$-dimensional complex Heisenberg Lie algebra
($N\geqslant 1$)
as in example \ref{exampleheisenberg}.
\\(i)
 $ZL^2_0(
{\mathcal{H}}_N,{\mathcal{H}}_N
 )$
 has basis
 $(x_{2N+1} \otimes (\omega^i \odot \omega^{j}))_{1 \leqslant i \leqslant j \leqslant 2N}$
 with
$(\omega^i)_{1\leqslant i \leqslant 2N+1}$ the dual basis to
$(x_i)_{1\leqslant i \leqslant 2N+1}$
($\odot $
stands for  the symmetric product
 $\omega^i \odot \omega^j =
 \omega^i \otimes \omega^j +
 \omega^j \otimes  \omega^i).$
 \\(ii)
 $$\dim ZL^2_0({\mathcal{H}}_N , {\mathcal{H}}_N ) = \dim B^2(
{\mathcal{H}}_N,{\mathcal{H}}_N
 ) = N(2N+1);$$
 $$\dim HL^2(
{\mathcal{H}}_N,{\mathcal{H}}_N
 ) = \dim Z^2(
{\mathcal{H}}_N,{\mathcal{H}}_N
 ) =
 \begin{cases}
 \frac{N}{3}(8N^2 +6N+1) & \text{ if } N\geqslant 2 \\
 8 &\text{ if } N=1\, .
 \end{cases}
 $$
\end{corollary}
\begin{proof}
(i) Follows from
$\ker{\mathcal{I}} =
S^2  \left(
\mathfrak{g}/\mathcal{C}^2 \mathfrak{g}
\right)^*
.$
\\(ii)
First
$\mathcal{H}_N$
is adjoint $ZL^2$-uncoupling since it is $\mathcal{I}$-null.
The result then follows  from the fact that (\cite{commalg})
 $ \dim B^2(
{\mathcal{H}}_N,{\mathcal{H}}_N
 ) = N(2N+1)$ and for $N\geqslant 2,$
 $ \dim H^2(
{\mathcal{H}}_N,{\mathcal{H}}_N
 ) = \frac{2N}{3}(4N^2-1).$
\end{proof}

\begin{example}
\rm
The case $N=1$ has been studied in \cite{fm}.  In that case,
\linebreak[4]
 $\dim ZL^2_0({\mathcal{H}}_1 , {\mathcal{H}}_1 ) = 3$ and the 3 Leibniz deformations
 are nilpotent, in contradistinction with the 5 Lie deformations. The
 authors completely describe a Leibniz versal deformation of the
 3-dimensional Heisenberg algebra.
\end{example}

\begin{example}
\rm
The 4-dimensional solvable "diamond" Lie algebra $\mathfrak{d}$ has basis
$(x_1,x_2,x_3,x_4)$ and
nonzero commutation relations (with anticommutativity)
\begin{equation}
\label{usual}
[x_1,x_2]=x_3,
[x_1,x_3]=-x_2,
[x_2,x_3]=x_4.
\end{equation}
The relations show that $\mathfrak{d}$ is an extension of the
one-dimensional abelian Lie algebra $\Cmath x_1$ by the Heisenberg algebra
$\mathfrak{n}_3$ with basis $x_2, x_3, x_4$.
It is also known as the Nappi-Witten Lie algebra \cite{nw} or the
central extension of the Poincar\'e Lie algebra in two dimensions. It
is a solvable quadratic Lie algebra, as admits a nondegenerate
bilinear symmetric
invariant form. Because of these properties, it plays an important role
in conformal field theory. We can use $\mathfrak{d}$ to construct a
Wess-Zumino-Witten model, which describes a homogeneous
four-dimensional Lorentz-signature space time \cite{nw}.
It is easy to check that $\mathfrak{d}$ is $\mathcal{I}$-exact.
In fact, one verifies that all other solvable 4-dimensional
Lie algebras are
 $\mathcal{I}$-null (for a list, see e.g. \cite{ovando}).

Consider $\mathfrak{d}$ as Leibniz algebra with a different basis
$\{e_1,~e_2,~e_3,~e_4\}$ over
 $\mathbb{C}$. Define a bilinear map $[~,~]: L\times L \longrightarrow L$
 by $[e_2,e_3]=e_1$, $[e_3,e_2]=-e_1$, $[e_2,e_4]=e_2$, $[e_4,e_2]=-e_2$,
 $[e_3,e_4]=e_2 -e_3$ and $[e_4,e_3]= e_3-e_2$, all other products of basis
  elements being $0$.

  \rm
We get a basis satisfying the usual commutation relations
(\ref{usual}) by letting
\begin{equation}
\label{gotousual}
x_1=ie_4, \;
  x_2=e_3, \;
  x_3=i(-e_2+e_3), \;
  x_4=ie_1.
\end{equation}

 One should mention that even though these two forms are equivalent
 over $\mathbb{C}$, they represent the two nonisomorphic real forms of the
 complex diamond algebra.

 We found that by considering Leibniz algebra deformation of
$\mathfrak{d}$ one gets more structures. Indeed it gives not only extra
stucture but also keeps track of Lie structures obtained by considering
Lie algebra  deformations. To get the precise deformations we need to
consider the  cohomology groups.

We compute cohomologies necessary for our purpose.
First consider the Leibniz cohomology space $HL^2(L;L)$.
Our computation consists of the following steps:\\
 (i) To determine a basis of the space of cocycles $ZL^2(L;L)$,\\
 (ii) to find out a basis of the coboundary space $BL^2(L;L)$,\\
 (iii) to determine the quotient space  $HL^2(L;L)$.\\

(i) Let $\psi$ $\in$ $ZL^2(L;L)$. Then $\psi :L\otimes L\longrightarrow L$ is a linear map and  $\delta \psi =0$, where
 \begin{equation*}
 \begin{split}
 \delta \psi(e_i, e_j, e_k)
 &=[e_i,\psi(e_j, e_k)]+[\psi (e_i, e_k), e_j]-[\psi(e_i, e_j), e_k] -\psi([e_i, e_j], e_k) \\
&~ +\psi(e_i,[e_j,e_k])+\psi([e_i, e_k], e_j) ~\mbox{for}~0\leq i,j,k \leq 4.
\end{split}
\end{equation*}
Suppose  $\psi(e_i,e_j)=\sum_{k=1} ^{4} a_{i,j}^{k} e_k$ where $a_{i,j}^{k} \in \mathbb C$
  ; for $1\leq i,j,k\leq 4$.
Since $\delta \psi =0$ equating the coefficients  of $e_1, e_2, e_3
~\mbox{and}~ e_4
$ in $\delta \psi(e_i, e_j, e_k)$ we get the following relations:
\begin{equation*}
\begin{split}
&(i)~ a_{1,1}^1=a_{1,1}^2 =a_{1,1}^3=a_{1,1}^4=a_{1,2}^1=a_{1,2}^3=a_{1,2}^4=0 ;\\
&(ii)~a_{1,3}^4=a_{1,4}^3=a_{1,4}^4=a_{2,1}^1=a_{2,1}^3=a_{2,1}^4=a_{2,2}^1=a_{2,2}^2=a_{2,2}^3=a_{2,2}^4=0;\\
&(iii)~a_{3,1}^4=a_{3,3}^2=a_{3,3}^3=a_{3,3}^4=a_{4,1}^3=a_{4,1}^4=a_{4,4}^2=a_{4,4}^3=a_{4,4}^4=0;\\
&(iv)~a_{1,2}^2=-a_{2,1}^2=a_{1,3}^2=-a_{1,3}^3=-a_{3,1}^2=a_{3,1}^3; \\
&(v)~a_{1,3}^1=-a_{3,1}^1=a_{1,4}^2=-a_{4,1}^2 ;\\
&(vi)~a_{2,3}^3=-a_{3,2}^3=-a_{2,4}^4=a_{4,2}^4;~~ a_{2,3}^4=-a_{3,2}^4;~~ a_{2,3}^2=-a_{3,2}^2;\\
&(vi)~a_{2,4}^1=-a_{4,2}^1;~a_{2,4}^2=-a_{4,2}^2;~a_{2,4}^3=-a_{4,2}^3;\\
&(vii)~a_{3,4}^1=-a_{4,3}^1;~a_{3,4}^2=-a_{4,3}^2;~a_{3,4}^3=-a_{4,3}^3;~a_{3,4}^4=-a_{4,3}^4\\
&(ix)~a_{3,4}^3=(a_{14}^1-a_{24}^2);~~a_{3,4}^4=(a_{14}^2+a_{23}^2)\\
&(x)a_{33}^1=\frac{1}{2}(a_{23}^1+a_{32}^1);~a_{41}^1=-(a_{14}^1+a_{23}^1+a_{32}^1).
\end{split}
\end{equation*}
 Therefore, in terms of the ordered basis
$\{e_i\otimes e_j\}_{1\leq i,j\leq 4}$ of $L\otimes L$ and $\{e_i\}_{1\leq i \leq 4}$ of $L$,transpose of the matrix corresponding to $\psi $ is of the form

$$M^t= \left( \begin{array}{llrr}
~0           & ~~0       &~~ 0          &~~  0       \\
~0           & ~~x_1     &~~ 0          &~~  0       \\
~x_2         & ~~x_1     & -x_1         & ~~  0      \\
~x_3         & ~~x_2     & ~~0          &~~  0       \\
~0           & -x_1      & ~~0        & ~~ 0     \\
~0           & ~~0       & ~~0          & ~~ 0       \\
~x_4         & ~~x_5     & ~~x_6        &~~x_7    \\
~x_8         & ~~x_9  & ~~x_{10}     & -x_6     \\
-x_2         &-x_1       & ~~x_1        &~~0  \\
x_{11}       & -x_5      & -x_6         &-x_7   \\
\frac{1}{2}(x_4+x_{11})         &~~ 0       & ~~0          & ~~0        \\
~x_{12}      &~~ x_{13}  & ~~(x_{3}-x_9) & ~~(x_2+x_5)  \\
-(x_4+x_3+x_{11})      &-x_2       &~~ 0          &~~0  \\
-x_{8}      &-x_{9}    &-x_{10}       &~~x_6       \\
-x_{12}      &-x_{13}    & -(x_{3}-x_9) & -(x_2+x_5) \\
~x_{14}      &~~ 0       & ~~0            & ~~ 0
\end{array}  \right).$$

\begin{equation*}
\begin{split}
&\mbox{where}~ x_1=a_{1,2}^2;~ x_2=a_{1,3}^1 ;~x_3=a_{1,4}^1;~
x_4=a_{2,3}^1;~x_5=a_{2,3}^2;~x_6=a_{2,3}^3;\\
&~x_7=a_{2,3}^4;~x_8=a_{2,4}^1;~x_9=a_{2,4}^2;x_{10}=a_{2,4}^3;~x_{11}=a_{3,2}^1;~x_{12}=a_{3,4}^1;\\
&~x_{13}=a_{3,4}^2~\mbox{and}~~x_{14}=a_{4,4}^1
\end{split}
\end{equation*} are in $\mathbb C$~.
Let $\phi_i \in ZL^2(L;L)$ for $1 \leq i\leq 14$, be the cocyle with
$x_i=1$ and $x_j=0$ for $i\neq j$ in the above matrix of $\psi$. It is
easy to check that $\{\phi_1,\cdots, \phi_{14}\}$ forms a basis of
$ZL^2(L;L)$.

(ii)
Let $ \psi_0 \in BL^2(L;L)$. We have $\psi_0=\delta g$ for some
$1$-cochain $g \in CL^1(L;L)=\text{Hom\,}(L;L)$. Suppose the matrix associated
to $\psi_0$ is same as the above matrix $M$.

Let  $g(e_i)=a_i ^1 e_1 +a_i ^2 e_2+a_i ^3 e_3+a_i^4 e_4$ for $i=1,2,3,4$.
The matrix associated to $g$ is given by
\begin{center} $\left(\begin{array}{llll}
a_1 ^1& a_2 ^1  & a_3 ^1 & a_4 ^1   \\
a_1 ^2& a_2 ^2  & a_3 ^2 & a_4 ^2    \\
a_1 ^3& a_2 ^3  & a_3 ^3 & a_4 ^3   \\
a_1 ^4& a_2 ^4  & a_3 ^4 & a_4 ^4
\end{array} \right).$ \end{center}
From the definition of coboundary we get $$\delta g(e_i,e_j)=[e_i,g(e_j)]+[g (e_i),e_j]-\psi([e_i,e_j])$$ for $0\leq i,j \leq 4$. The transpose matrix of $\delta g$ can be written as

$$ \left( \begin{array}{llrr}
~0                         & ~0                                &~~ 0               &~~  0              \\
-a_{1}^3                   & -a_{1}^4                          &~~ 0               &~~  0              \\
~a_{1}^2                   & -a_{1}^4                          & a_{1}^4           & ~~  0             \\
~0                         & (a_{1}^2+a_{1}^3)                 &-a_{1}^3           &~~  0               \\
~a_{1}^3                   & a_{1}^4                           & ~~0               & ~~ 0               \\
~0                         & ~~0                               & ~~0               & ~~ 0               \\
-(a_{1}^1-a_{2}^2-a_{3}^3) & -(a_{1}^2+a_{2}^4-a_{3}^4)        & -(a_{1}^3-a_{2}^4 &-a_{1}^4            \\
-(a_{2}^1-a_{4}^3)         & (a_{2}^3+a_{4}^4)                 & -2a_{2}^3         & -a_{2}^4            \\
-a_{1}^2                   &a_{1}^4                            & -a_{1}^4          &~~0                 \\
~(a_{1}^1-a_{2}^2-a_{3}^3) & (a_{1}^2+a_{2}^4-a_{3}^4)         & (a_{1}^3-a_{2}^4) &a_{1}^4             \\
~0                         &~~ 0                               & ~~0               & ~~0                 \\
-(a_{2}^1-a_{3}^1+a_{4}^2) &-(a_{2}^2-2a_{3}^2-a_{3}^3-a_{4}^4)& -(a_{2}^3+a_{4}^4)& -(a_{2}^4-a_{3}^4)  \\
~0                         &-(a_{1}^2+a_{1}^3)                 &a_{1}^3            &~~0                  \\
~(a_{2}^1-a_{4}^3)         &-(a_{2}^3+a_{4}^4)                 &2a_{2}^3           &~a_{2}^4             \\
~(a_{2}^1-a_{3}^1+a_{4}^2) &(a_{2}^2-2a_{3}^2-a_{3}^3-a_{4}^4) & (a_{2}^3+a_{4}^4) & (a_{2}^4-a_{3}^4)   \\
~0                         &~~ 0                               & ~~0               & ~~ 0
\end{array}  \right).$$

Since $\psi_0=\delta g$ is also a cocycle in $CL^2(L;L)$, comparing matrices $\delta g$ and $M$ we conclude that the transpose matrix of $\psi_0$  is of the form

$$M^t= \left( \begin{array}{llrr}
~0           & ~~0       &~~ 0          &~~  0       \\
~0           & ~~x_1     &~~ 0          &~~  0       \\
~x_2         & ~~x_1     & -x_1         & ~~  0      \\
~0           & ~~x_2     & ~~0          &~~  0       \\
~0           & -x_1      & ~~0          & ~~ 0     \\
~0           & ~~0       & ~~0          & ~~ 0       \\
~x_4         & ~~x_5     & ~~x_6        &~~x_1    \\
~x_8         & ~~x_9     & ~~x_{10}     & -x_6     \\
-x_2         &-x_1       & ~~x_1        &~~0  \\
-x_{4}       & -x_5      & -x_6         &-x_1   \\
~0           &~~ 0       & ~~0         & ~~0        \\
~x_{12}      &~~ x_{13}  &-x_9          & ~~(x_2+x_5)  \\
~0           &-x_2       &~~ 0          &~~0  \\
-x_{8}       &-x_{9}    &-x_{10}       &~~x_6       \\
-x_{12}      &-x_{13}    & ~~x_9       & -(x_2+x_5) \\
~0           &~~ 0       & ~~0            & ~~ 0
\end{array}  \right).$$

Let ${\phi_i}^\prime \in BL^2(L;L)~\mbox{for}~i=1,2,4,5,6,8,9,10,12,13$ be the coboundary with $x_i=1$ and $x_j=0$ for $i\neq j$ in the above matrix of
$\psi_0$. It follows that $\{\phi_1^\prime,\phi_2^\prime,\phi_4^\prime,\phi_5^\prime,\phi_6^\prime,\phi_{8}^\prime,\phi_{9}^\prime,\phi_{10}^\prime,\phi_{12}^\prime,\phi_{13}^\prime\}$ forms a basis of the coboundary space $BL^2(L;L)$.

(iii)
It is  straightforward to check that  $$\{[\phi_3],[\phi_7],[\phi_{11}],[\phi_{14}]\}$$
span $HL^2(L;L)$ where $[\phi_i]$ denotes the cohomology class represented by the cocycle $\phi_i$.

Thus  $\dim(HL^2(L;L))=4$.

The representative cocycles of the cohomology classes forming a basis
of $HL^2(L;L)$ are given explicitely as the following.
\begin{equation*}
 \begin{split}
&(1)~\phi_{3} : \phi_3(e_1,e_4)=e_1,~ \phi_3(e_4,e_1)=-e_1;~\phi_3(e_3,e_4)=e_3;~\phi_3(e_4,e_3)=-e_3;\\
&(2)~\phi_{7} : \phi_{7}(e_2,e_3)=e_4,~ \phi_{7}(e_3,e_2)=-e_4;\\
&(3)~\phi_{11} : \phi_{11}(e_3,e_2)=e_1,~\phi_{11}(e_3,e_3)=\frac{1}{2}e_1,~\phi_{11}(e_4,e_1)=-e_1;\\
&(4)~\phi_{14}:\phi_{14}(e_4,e_4)=e_1.
\end{split}
\end{equation*}
Here $\phi_3$ and $\phi_{7}$ are skew-symmetric, so $\phi_i \in
Hom(\Lambda^2 L ; L) \subset Hom(L^{\otimes 2};L)$ for $i=3$ and $7$.

Consider, $\mu_i=\mu_0 +t \phi_i$ for $i=3,7,11,14$, where $\mu_0$
denotes the original bracket in $L$.

This gives $4$ non-equivalent infinitesimal deformations of the Leibniz
bracket $\mu_0$ with $\mu_3$ and $\mu_{7}$ giving the Lie algebra
structure on $L[[t]]/<t^2>$.

Now we have to compute the Massey brackets $[\phi_i, \phi_j]$ which are
responsible for obstructions to extend infinitesimal deformations.
We find
$$
[\phi_3, \phi_3] = 0, \quad [\phi_7, \phi_7] = 0.
$$
That means that the two infinitesimal Lie deformations can be extended
to real deformations, with the new nonzero brackets (and their
anticommutative version)

The first of the deformations represents a 2-parameter projective family
$d(\lambda,\mu)$, for which each projective
parameter $(\lambda,\mu)$ defines a nonisomorphic
Lie algebra (in fact, the diamond algebra is a member of this family
with $(\lambda,\mu) = (1,-1))$:
\begin{align*}
[e_2,e_3]_{\lambda,\mu} &= e_1 \\
[e_2,e_4]_{\lambda,\mu} &= \lambda e_2 \\
[e_3,e_4]_{\lambda,\mu} &= e_2 + {\mu}e_3  \\
[e_1,e_4]_{\lambda,\mu} &= (\lambda+\mu)e_1.
\end{align*}
The second deformation,
\begin{align*}
[e_2,e_3]_t &= e_1 + te_4 \\
[e_2,e_4]_t &= e_2 \\
[e_3,e_4]_t &= e_2 - e_3
\end{align*}
is isomorphic to $\mathfrak{sl}(2,\Cmath) \oplus
\Cmath$ for every nonzero value of $t$, see \cite{fp}.
\medskip
Furthermore, we also have $[\phi_{14},\phi_{14}] = 0$ which means that
$\phi_{14}$ defines a real Leibniz deformation:
\begin{align*}
[e_2,e_3]_t &= e_1 \\
[e_2,e_4]_t &= e_2 \\
[e_3,e_4]_t &= e_2 - e_3 \\
[e_4,e_4]_t &= te_1.
\end{align*}
We note that this Leibniz algebra is not nilpotent.

For the bracket $[\phi_{11},\phi_{11}]$ we get a nonzero 3-cocycle, so the
infinitesimal Leibniz deformation with infinitesimal part being
$\phi_{11}$ can not be extended even to the next order.

\medskip
The nontrivial mixed brackets $[\phi_i,\phi_j]$ determine relations on
the base of versal deformation.

\medskip
Among the six possible cases
$[\phi_{3},\phi_{11}]$, $[\phi_{3},\phi_{14}]$ and $[\phi_{11},\phi_{14}]$
are nontrivial $3$-cocycles, the others are represented by
$3$-coboundaries.

Thus we need to check the Massey $3$-brackets which are defined, namely

$<\phi_{3},\phi_{3},\phi_{7}>$

$<\phi_{3},\phi_{7},\phi_{7} >$

$<\phi_{7},\phi_{7},\phi_{11} >$

$<\phi_{7},\phi_{7},\phi_{14} >$

$<\phi_{7},\phi_{14},\phi_{14} >$

In these five possible Massey $3$-brackets, only $<\phi_{3},\phi_{3},\phi_{7} >$
 is represented by nontrivial cocycle.

So we now proceed to compute the possible Massey $4$-brackets. We get
that four of them are nontrivial:

$<\phi_{3},\phi_{7},\phi_{7},\phi_{11} >$

$<\phi_{3},\phi_{7},\phi_{7},\phi_{14} >$

$<\phi_{7},\phi_{7},\phi_{14},\phi_{11} >$

$<\phi_{7},\phi_{7},\phi_{14},\phi_{14} >$.

At the next step, we get that all the 5-order Massey products are
either not defined or are trivial.

So we can write the versal Leibniz deformation of our Lie algebra:
\begin{align*}
&[e_1, e_2]_v = [e_2, e_1]_v = [e_1, e_3]_v = [e_3, e_1]_v = 0,\\
&[e_1, e_4]_v = te_1,\\
&[e_4, e_1]_v = -(t + u)e_1,\\
&[e_2, e_3]_v = e_1 + se_4, \\
&[e_3, e_2]_v = (u - 1)e_1 -se_4 ,\\
&[e_2, e_4]_v = e_2,\\
&[e_4, e_2] = -e_2,\\
&[e_3, e_4]_v = e_2 + (t - 1)e_3,\\
&[e_4, e_3]_v = -e_2 + (1 - t)e_3,\\
&[e_1, e_1]_v = [e_2, e_2]_v = 0,\\
&[e_3, e_3]_v = 1/2ue_1,\\
&[e_4, e_4]_v = we_1.
\end{align*}
The base of the versal deformation is
$$
 \Cmath[[t,s,u,w]]/\{tu,tw,uw; t^2s; ts^2u,ts^2w,s^2uw,s^2w^2\}.
$$
\end{example}

\rm
\begin{example}
\rm
The quadratic $5$-dimensional nilpotent Lie algebra
$\mathfrak{g}_{5,4}$ \cite{ART}
has commutation relations
$[x_1,x_2]=x_3,$
$[x_1,x_3]=x_4,$
$[x_2,x_3]=x_5.$

This is an extension of the trivial Lie algebra $\Cmath x_1$ by the
 4-dimensional Lie algebra $\Cmath x_4 \times \mathfrak{n}_3$
($\mathfrak{n}_3$ the 3-dimensional Heisenberg Lie algebra
$[x_2,x_3]=x_5$). As it is moreover the only 5-dimensional indecomposable nilpotent Lie algebra
 which is not $\mathcal{I}$-null, it can be considered as a 5-dimensional
 analogue of the diamond algebra $\mathfrak{d}$.

Let us first compute its trivial Leibniz cohomology.
We here denote simply $d$ for $d_{\Cmath}$, and
 $\omega^{i,j}$ for  $ \omega^i\wedge   \omega^j$
 (see also \cite{hindawi},\cite{Magnin1}).
 \\
$B^2(\mathfrak{g},\Cmath) = \langle
d\omega^3=-\omega^{1,2}, d\omega^4=-\omega^{1,3}, d\omega^5=-\omega^{2,3}\rangle,$
$\dim Z^2(\mathfrak{g},\Cmath) = 6,$
$\dim H^2(\mathfrak{g},\Cmath) = 3,$
$Z^2(\mathfrak{g},\Cmath) = \langle
\omega^{1,4},
\omega^{2,5},
\omega^{1,5}+
\omega^{2,4}\rangle \oplus  B^2(\mathfrak{g},\Cmath),$
$\dim{ZL^2_0(\mathfrak{g},\Cmath)}=3,$
$ZL^2_0(\mathfrak{g},\Cmath)
(\cong \ker{\mathcal{I}})
= \langle
\omega^1 \otimes \omega^1,
\omega^1 \odot \omega^2,
\omega^2 \otimes \omega^2 \rangle,$
$\dim{ZL^2(\mathfrak{g},\Cmath) }= 10,$
$\dim{HL^2(\mathfrak{g},\Cmath) }= 7,$ and
\begin{eqnarray*}
ZL^2(\mathfrak{g},\Cmath)& =&
Z^2(\mathfrak{g},\Cmath) \oplus
ZL^2_0(\mathfrak{g},\Cmath)
\oplus \Cmath g_1,\\
HL^2(\mathfrak{g},\Cmath) &=&
H^2(\mathfrak{g},\Cmath) \oplus
ZL^2_0(\mathfrak{g},\Cmath)
\oplus \Cmath g_1
\end{eqnarray*}
with
$g_1= B + \omega^{1,5}$
and $B= \omega^1 \odot \omega^5 -  \omega^2 \odot \omega^4  +\omega^3 \otimes \omega^3.$
(Here
$\text{Im\,}{\mathcal{I}} = \Cmath I_B = \Cmath d\omega^{1,5}$
and
$\text{Im\,}{\mathcal{I}} \cap B^3(\mathfrak{g},\Cmath)=
\text{Im\,}{\mathcal{I}} $
is one-dimensional.)
${\mathfrak{g}}_{5,4}$  is not trivial $ZL^2$-uncoupling
(hence not adjoint $ZL^2$-uncoupling either),
and $g_1$ is a coupled Leibniz 2-cocycle.\\

Now let us turn to the adjoint Leibniz cohomology, which represents
nonequivalent infinitesimal Leibniz deformations.\\
$\dim{Z^2(\mathfrak{g},\mathfrak{g})} = 24;$
$ZL^2_0(\mathfrak{g},\mathfrak{g})= \mathfrak{c}\, \otimes \ker{\mathcal{I}}$ has dimension 6,
$\dim{ZL^2(\mathfrak{g},\mathfrak{g})}=32,$
\begin{eqnarray*}
ZL^2(\mathfrak{g},\mathfrak{g})&=&
Z^2(\mathfrak{g},\mathfrak{g})
\oplus
ZL^2_0(\mathfrak{g},\mathfrak{g})
\oplus \Cmath G_1
\oplus \Cmath G_2,\\
HL^2(\mathfrak{g},\mathfrak{g})&=&
H^2(\mathfrak{g},\mathfrak{g})
\oplus
ZL^2_0(\mathfrak{g},\mathfrak{g})
\oplus \Cmath G_1
\oplus \Cmath G_2,
\end{eqnarray*}
where $G_1,G_2$ are the following Leibniz 2-cocycles, each of which is coupled:
\begin{eqnarray*}
G_1&=&
x_5 \otimes ( B +  \omega^{1,5})\\
G_2&=&
x_4 \otimes ( B +   \omega^{1,5})
\end{eqnarray*}
Here $H^2(\mathfrak{g},\mathfrak{g})$ has dimension 9. \\
Of course, these spaces are huge to compute, but we would like to point
out some structural similarity with the diamond algebra.

\medskip

One may observe that the coupled cocycle $\phi_{11}$ of $\mathfrak{d}$
reads
in the basis (\ref{gotousual})
$$\phi_{11}= -ix_4 \otimes (C -\omega^{2,3}+\omega^{1,4})$$
with $C=
\omega^{1} \odot \omega^{4}
+\omega^{2} \otimes \omega^{2}
+\omega^{3} \otimes \omega^{3}
$
the non degenerate invariant bilinear form,
a similarity with $G_1, G_2.$
The similarity extends to the fact that $G_1,G_2$ cannot be extended
to the second level.

\bigskip
As of Lie deformations, $\mathfrak{g}_{5,4}$ has a number of
deformations. Without identifying all of them, we list some:

1. A three-parameter solvable projective family $d(p:q:r)$ where $\mathfrak{g}_{5,4}$
belongs (it is its nilpotent element, with $p=q=r=0$) with nonzero brackets
\begin{align*}
[x_3,x_4]_{p,q,r}&=x_2 \\
[x_1,x_5]_{p,q,r}&=rx_1 \\
[x_2,x_5]_{p,q,r}&=(p+q)x_2 \\
[x_3,x_5]_{p,q,r}&=px_3 + x_1 \\
[x_4,x_5]_{p,q,r}&=x_3 + qx_4.
\end{align*}

2. A solvable Lie algebra with nonzero brackets
\begin{align*}
[x_3,x_4]&=2x_4\\
[x_3,x_5]&=-2x_5\\
[x_4,x_5]&=x_3\\
[x_1,x_2]&=x_1.
\end{align*}

3. Another solvable Lie algebra with nonzero brackets
\begin{align*}
[x_3,x_4]&=2x_4\\
[x_3,x_5]&=-2x_5\\
[x_4,x_5]&=x_3\\
[x_1,x_3]&=x_1\\
[x_2,x_5]&=x_1\\
[x_2,x_3]&=-x_2\\
[x_1,x_4]&=x_2.
\end{align*}

4. A 2-parameter solvable projective family with nonzero brackets
\begin{align*}
[x_2,x_5]_{p,q}&=x_1+px_2\\
[x_3,x_5]_{p,q}&=x_2+qx_3\\
[x_4,x_5]_{p,q}&=x_3+(p+q)x_4\\
[x_1,x_5]_{p,q}&=(p+q)x_1\\
[x_2,x_3]_{p,q}&=pqx_1\\
[x_2,x_4]_{p,q}&=qx_1\\
[x_3,x_4]_{p,q}&=x_1.
\end{align*}

5. Another 2-parameter solvable projective family with nonzero brackets
\begin{align*}
[x_3,x_4]_{p,q}&=x_2\\
[x_2,x_5]_{p,q}&=(p+q)x_2\\
[x_3,x_5]_{p,q}&=x_1+px_3\\
[x_4,x_5]_{p,q}&=x_3+qx_4\\
[x_1,x_5]_{p,q}&=(q+2p)x_1\\
[x_2,x_3]_{p,q}&=(p-q)x_1\\
[x_2,x_4]_{p,q}&=x_1.
\end{align*}
\end{example}


\end{document}